\documentclass[11pt,reqno]{amsart}

\usepackage[utf8]{inputenc}
\usepackage{amsmath,amssymb,amsthm,mathtools}
\usepackage{graphicx}
\usepackage{enumitem}
\usepackage{a4wide}
\usepackage[numbers,sort&compress]{natbib}
\usepackage[hidelinks]{hyperref}

\newcommand{\norm}[1]{\left\lVert#1\right\rVert}
\newcommand{\A}{\mathcal{A}}
\newcommand{\mean}[1]{\langle#1\rangle}

\theoremstyle{plain}
\newtheorem{theorem}{Theorem}[section]
\newtheorem{lemma}[theorem]{Lemma}
\newtheorem{proposition}[theorem]{Proposition}
\newtheorem{corollary}[theorem]{Corollary}

\theoremstyle{definition}
\newtheorem{example}[theorem]{Example}

\theoremstyle{remark}
\newtheorem{remark}[theorem]{Remark}

\numberwithin{equation}{section}
\title[Length bounds and rigidity]{Length bounds and rigidity for ancient curve shortening flow on warped product surfaces}

\author{Beomseok Goh}
\address{Department of Mathematics, Konkuk University, Seoul 05029, Republic of Korea}
\email{zclub8740@gmail.com}

\author{Taehun Lee}
\address{Department of Mathematics, Konkuk University, Seoul 05029, Republic of Korea}
\email{taehun@konkuk.ac.kr}


\begin{document}

\begin{abstract}
We study ancient solutions to curve shortening flow on a class of warped product surfaces.
For closed embedded ancient solutions whose curves are not null-homotopic, we prove that graphicality at every time is equivalent to a uniform length bound as $t\to-\infty$, and that these properties hold whenever the length is $o(\sqrt{-t})$.
Under an additional positivity condition on the linearized height equation, every such solution with bounded length is a family of slices.
We derive backward convergence to a limiting circle from the length bound and prove slice rigidity by an $L^2$ estimate for the height perturbation.
\end{abstract}

\maketitle

\section{Introduction}

A family of closed curves $\Gamma_t$ in a Riemannian surface \((M,g)\) is a solution to curve shortening flow if it has a parametrization $F(x,t)$, called a normal parametrization, such that
\begin{equation}\label{eq:csf}
    \frac{\partial F}{\partial t}=\kappa\mathbf{N}=\nabla_\mathbf{T}\mathbf{T},
\end{equation}
where \(\mathbf{T}\) is the unit tangent vector, \(\mathbf{N}\) is a chosen unit normal, and \(\kappa\) is the corresponding signed geodesic curvature.
Curve shortening flow is the negative gradient flow of length.
For classical results on plane curves, see \cite{Gage84,GH86,Grayson87}, and for Riemannian surfaces, see \cite{Grayson89,Gage90,Angenent90}.

We consider this flow on warped product surfaces $M=S^1\times(-\infty,a)$ with $a<\infty$ and metric $g=dz^2+r(z)^2\,d\theta^2$.
A solution is called \emph{ancient} if it is defined for \(t\in(-\infty,T)\).
We relate length bounds, graphicality, and slice rigidity for these solutions.

In the plane, compact convex ancient solutions are shrinking circles or Angenent ovals \cite{DHS10,Angenent92}, and the complete convex case is classified in \cite{BLT20}.
Angenent and You \cite{AY21} constructed nonconvex ancient solutions, and a compact embedded example was constructed in \cite{ZOKA23}.
Embedded ancient solutions with finite entropy have a unique tangent flow at infinity \cite{CSSZ25} and are classified in \cite{CSSZ26}.

On the round sphere, Bryan and Louie \cite{BL16} proved that a convex closed embedded ancient solution is either a stationary equator or a family of shrinking geodesic circles.
Choi and Mantoulidis \cite{CM22} obtained the same conclusion without convexity, assuming instead that the length stays bounded as $t\to-\infty$.
In a general setting, they constructed families of ancient mean curvature flows converging to closed embedded minimal submanifolds.
They also proved that an ancient flow belongs to the corresponding family, up to time translation, if it stays close to the limiting submanifold and its distance to it is integrable in time.
For ancient mean curvature flows of hypersurfaces, see \cite{ADS20,BC19,HS15} for convex rigidity results and \cite{CHL25,MP21} for constructions converging to unstable minimal hypersurfaces in Euclidean space.

Zhou \cite{Zhou17} studied graphical curve shortening flow in warped products.
Fujihara proved that curve shortening flow on warped product surfaces preserves graphicality for arbitrary warping functions \cite{Fujihara24}.
Under suitable assumptions, he also proved that an embedded closed curve that is not null-homotopic becomes graphical after finite time \cite{Fujihara25}.
Borisenko and Miquel \cite{BM12} studied mean curvature flow of graphs in warped products.
For volume preserving flows and constant mean curvature hypersurfaces in warped products, see \cite{GLW19} and \cite{Brendle13}, respectively.

The warping function $r$ is assumed smooth and positive on $I:=(-\infty,a)$.
Unless otherwise stated, we assume that $r'>0$ and that $r$, $r'$, and $r''$ are uniformly bounded on $I$.
We write
\[
f(z):=\frac{r'(z)}{r(z)}.
\]

Since \(r\) is increasing and bounded, it has a finite positive limit
\[
r(a):=\lim_{z\to a^-}r(z)>0.
\]
The boundedness of \(r''\) also gives a one-sided limit \(r'(a)\) and a \(C^{1,1}\) extension of \(r\) to each interval \([b,a]\), \(b<a\).
We do not assume that \(r''\) has a limit at \(a\) or that higher derivatives are uniformly bounded there.

We orient every closed embedded curve in $M$ that is not null-homotopic so that its projection to $S^1$ has degree one.
We call such a curve \emph{graphical} if it can be written as
\[
\Gamma=\{(\theta,z(\theta)):\theta\in S^1\}
\]
for a smooth \(2\pi\)-periodic function \(z\).
Unless otherwise specified, we consider smooth, closed, embedded ancient solutions whose curves are not null-homotopic in $M$.
A \emph{slice} is a curve of the form \(S^1\times \{Z\}\).

By Lemmas \ref{lem:slice} and \ref{lem:zmax} below, a closed ancient solution exists in $M$ if and only if $r'(a)=0$.
We denote by $L(t)$ the length of $\Gamma_t$.

Our first main result is as follows.

\begin{theorem}\label{thm:main}
    Let \(\{\Gamma_t\}_{t\in(-\infty,T)}\) be a smooth, closed, embedded ancient solution to curve shortening flow whose curves are not null-homotopic in \(M\).
    Then the following conditions are equivalent:
    \begin{enumerate}[label=\textup{(\roman*)}]
        \item\label{main:ii}
        For every \(t<T\), the curve \(\Gamma_t\) is graphical over \(S^1\).
        \item\label{main:iii}
        The length is uniformly bounded as \(t\to-\infty\).
        \item\label{main:iv}
        The length satisfies
        \[
        L(t)=o(\sqrt{-t})\qquad \text{as}\quad t\to-\infty.
        \]
    \end{enumerate}
    Under any of these conditions, $\Gamma_t$ converges in $C^1$ to $S^1\times\{a\}$ and $L(t)\to2\pi r(a)$ as $t\to-\infty$.
\end{theorem}

For a graphical solution, write $z=m+u$, where $m(t)$ is the angular mean of the height.
The linearized equation for the perturbation $u$ at a slice is
\[
\partial_t u=r(m)^{-2}u_{\theta\theta}-f'(m)u.
\]
The first nonconstant Fourier mode has decay rate $r(m)^{-2}+f'(m)$.
Set
\[
\mu:=\liminf_{z\to a^-}(r(z)^{-2}+f'(z)).
\]
For the round spherical metric, $r(z)=r(a)\sin(z/r(a))$ in meridian coordinates near the equator, so $r(z)^{-2}+f'(z)\equiv0$ and $\mu=0$.

Let $K=-r''/r$ be the Gaussian curvature of $M$ (see Lemma \ref{lem:conn}), so that $f'=-K-f^2$.
If $r'(a)=0$, which is the case whenever a closed ancient solution exists, then $\mu>0$ is equivalent to
\[
\limsup_{z\to a^-}K(z)<\frac{1}{r(a)^2}.
\]

Our second main result gives slice rigidity.

\begin{theorem}\label{thm:rig}
    Assume $\mu>0$.
    Let $\{\Gamma_t\}_{t\in(-\infty,T)}$ be a smooth, closed, embedded ancient solution to curve shortening flow whose curves are not null-homotopic in $M$.
    If $L(t)$ is uniformly bounded as $t\to-\infty$, then the solution is a family of slices
    \[
    \Gamma_t=S^1\times\{Z(t)\},\qquad Z'(t)=-\frac{r'(Z(t))}{r(Z(t))},
    \]
    with $Z(t)\to a$ as $t\to-\infty$ and $L(t)=2\pi r(Z(t))<2\pi r(a)$ for every $t<T$.
\end{theorem}

By Theorem \ref{thm:main}, the length bound in Theorem \ref{thm:rig} can be replaced by graphicality at every time or by $L(t)=o(\sqrt{-t})$ as $t\to-\infty$.

\begin{example}
    Let \(r(z)=\exp(-c(z-a)^2)\) for \(z<a\), where \(c>0\).
Then \(f(z)=2c(a-z)>0\), \(r(a)=1\), and
    \[
    \mu=1-2c.
    \]
    The function $r$ is increasing and bounded in $C^2(I)$ for every $c>0$.
For $0<c<1/2$, we have $\mu>0$.
    The slices \(S^1\times\{Z(t)\}\) with \(Z(t)=a-A_0e^{2ct}\), \(A_0>0\), are ancient solutions.
\end{example}

For Theorem \ref{thm:main}, the length bound confines the flow to a fixed strip at sufficiently negative times, and decay of the curvature energy gives subsequential backward convergence to a closed geodesic, as in the proof of \cite[Corollary~1.3]{CM22}.
Clairaut's relation identifies this geodesic with $S^1\times\{a\}$, yielding graphicality and $C^1$ convergence (Proposition \ref{prop:gph}).
Conversely, a bounded monotone weighted area gives backward convergence of the height and a sequence of times with vanishing total variation of the height, which yield the length bound (Proposition \ref{prop:l-gph}).
If some $\Gamma_t$ is not graphical, then $L(t)^2$ grows at least linearly in $-t$ (Proposition \ref{prop:subsqrt}).

For Theorem \ref{thm:rig}, the assumption $\mu>0$ and the Wirtinger inequality give a positive decay rate for the $L^2$ energy of the height perturbation (Proposition \ref{prop:a-dec}).
Since this energy also tends to zero as $t\to-\infty$, the perturbation vanishes at sufficiently negative times, and uniqueness extends the slice property to all times.

\section{Preliminaries}\label{sec:pre}

\begin{lemma}\label{lem:conn}
Let $(M,g)$ be the warped product surface with metric $g=r(z)^2\,d\theta^2+dz^2$, where $r$ is smooth and positive.
The nonzero Christoffel symbols are
\begin{equation}
    \Gamma_{\theta\theta}^z=-rr',\qquad \Gamma_{\theta z}^\theta=\Gamma_{z\theta}^\theta=\frac{r'}{r}=f.
\end{equation}
The covariant derivatives of the coordinate vector fields are
\begin{equation}\label{eq:conn}
\begin{aligned}
    \nabla_{\partial_\theta}\partial_\theta&=-rr'\partial_z,\\
    \nabla_{\partial_\theta}\partial_z=\nabla_{\partial_z}\partial_\theta&=f\partial_\theta,\qquad \nabla_{\partial_z}\partial_z=0.
\end{aligned}
\end{equation}
Moreover, the Gaussian curvature is $K=-r''/r$.
\end{lemma}

\begin{proof}
The metric components are $g_{\theta\theta}=r^2$, $g_{zz}=1$, and $g_{\theta z}=0$.
The Christoffel formula gives
\begin{equation}\label{eq:chr}
    \Gamma_{\theta\theta}^z=-\tfrac12\partial_zg_{\theta\theta}=-rr',\qquad \Gamma_{\theta z}^\theta=\Gamma_{z\theta}^\theta=\tfrac12g^{\theta\theta}\partial_zg_{\theta\theta}=\frac{r'}{r},
\end{equation}
and the remaining symbols vanish.
The covariant derivatives follow from $\nabla_{\partial_i}\partial_j=\Gamma_{ij}^k\partial_k$.
Since $[\partial_z,\partial_\theta]=0$ and $f=r'/r$,
\begin{equation}\label{eq:rm}
\begin{aligned}
    R(\partial_z,\partial_\theta)\partial_\theta
    &=\nabla_{\partial_z}(-rr'\partial_z)-\nabla_{\partial_\theta}(f\partial_\theta)\\
    &=-(r'^2+rr'')\partial_z+rr'f\partial_z=-rr''\partial_z.
\end{aligned}
\end{equation}
Therefore,
\begin{equation}\label{eq:k}
    K=r^{-2}g(R(\partial_z,\partial_\theta)\partial_\theta,\partial_z)=-\frac{r''}{r}.
\end{equation}
\end{proof}

For a slice, the connection formulas reduce curve shortening flow to an ordinary differential equation for its height.

\begin{lemma}\label{lem:slice}
For $z_0<a$, the slice $S^1\times\{z_0\}$ evolves by curve shortening flow as $\Gamma_t=S^1\times\{Z(t)\}$, where $Z(0)=z_0$ and
\begin{equation}\label{eq:slice}
    Z'(t)=-f(Z(t)).
\end{equation}
This solution extends to all negative times if and only if $r'(a)=0$.
In this case, $Z(t)\to a$ as $t\to-\infty$.
\end{lemma}

\begin{proof}
For a slice, $\mathbf T=r(Z)^{-1}\partial_\theta$, and the connection formula in Lemma \ref{lem:conn} gives
\begin{equation}
    \nabla_{\mathbf T}\mathbf T=-\frac{r'(Z)}{r(Z)}\partial_z.
\end{equation}
The parametrization $F(\theta,t)=(\theta,Z(t))$ has normal velocity $Z'(t)\partial_z$, so \eqref{eq:csf} gives \eqref{eq:slice}.

Separating variables in \eqref{eq:slice} and using $Z(0)=z_0$, we obtain $-t=\int_{z_0}^{Z(t)}\frac{d\zeta}{f(\zeta)}$.
Set
\begin{equation}
    \tau_a:=\int_{z_0}^a\frac{d\zeta}{f(\zeta)}.
\end{equation}
Thus the solution extends to all negative times if and only if $\tau_a=\infty$.

If $r'(a)>0$, then $f(z)\to r'(a)/r(a)>0$, so $\tau_a<\infty$.
If $r'(a)=0$, boundedness of $r''$ gives $0<r'(z)\le C(a-z)$ for $z_0\le z<a$.
Since $r(z)\ge r(z_0)>0$, there is a constant $C_0>0$ such that
\begin{equation}
    0<f(z)\le C_0(a-z),\qquad \tau_a\ge\frac1{C_0}\int_{z_0}^a\frac{d\zeta}{a-\zeta}=\infty.
\end{equation}

In this case, the integral identity also gives $Z(t)\to a$ as $t\to-\infty$ since the integral is finite for every upper endpoint below $a$.
\end{proof}

\begin{figure}
    \centering
    \includegraphics[width=0.8\linewidth]{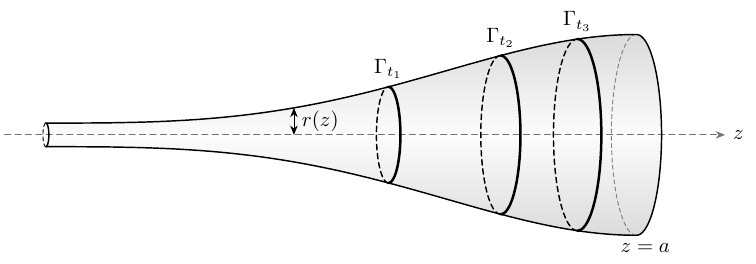}
    \caption{The slices $\Gamma_t=S^1\times\{Z(t)\}$ of Lemma \ref{lem:slice} at times $t_1>t_2>t_3$.}
    \label{fig:slice}
\end{figure}

The condition $r'(a)=0$ is equivalent to the geodesic curvature of the slices tending to zero at the upper end (see Figure \ref{fig:slice}).

Let $\pi_I:M\to I$ be the projection given by $\pi_I(\theta,z)=z$.
Let $F(x,t)$ be a normal parametrization satisfying \eqref{eq:csf}, where $x\in S^1$.
Write $h(x,t)=\pi_I(F(x,t))$, with $h_t$ taken at fixed $x$.
Let $s$ denote arclength at each fixed time, so $\partial_s=|\partial_xF|_g^{-1}\partial_x$, and let $\mathbf T=\partial_sF$.
Choose the unit normal $\mathbf{N}$ so that $(\mathbf T,\mathbf{N})$ has the same orientation as the orthonormal frame $(r^{-1}\partial_\theta,\partial_z)$.
Set
\begin{equation}
    \Theta:=\langle\mathbf{N},\partial_z\rangle.
\end{equation}
Write $\Delta=\partial_s^2$ for the intrinsic Laplacian along $\Gamma_t$.
The function $\Theta$ may have either sign.

Using a local lift $\vartheta$ of the angular coordinate, write
\begin{equation}
    F(x,t)=(\vartheta(x,t),h(x,t)).
\end{equation}

\begin{lemma}\label{lem:height}
In the normal parametrization, the height function $h$ satisfies
\begin{equation}\label{eq:h-norm}
    h_t=h_{ss}-f(h)\Theta^2.
\end{equation}
\end{lemma}

\begin{proof}
In a local angular lift, write
\begin{equation}
    \mathbf{T}=\vartheta_s\partial_\theta+h_s\partial_z,\qquad r(h)^2\vartheta_s^2+h_s^2=1.
\end{equation}
The compatible unit normal is
\begin{equation}\label{eq:n-fr}
    \mathbf{N}=-\frac{h_s}{r(h)}\partial_\theta+r(h)\vartheta_s\partial_z,\qquad \Theta=r(h)\vartheta_s.
\end{equation}
Since $\nabla\pi_I=\partial_z$, Lemma \ref{lem:conn} gives
\begin{equation}
    \operatorname{Hess}_M\pi_I(\mathbf{T},\mathbf{T})=\langle\nabla_\mathbf{T}\partial_z,\mathbf{T}\rangle=f(h)r(h)^2\vartheta_s^2=f(h)\Theta^2.
\end{equation}
Using $\nabla_\mathbf{T}\mathbf{T}=\kappa\mathbf{N}$, we obtain
\begin{equation}
    h_{ss}=\operatorname{Hess}_M\pi_I(\mathbf{T},\mathbf{T})+\langle\nabla_\mathbf{T}\mathbf{T},\partial_z\rangle=f(h)\Theta^2+\kappa\Theta.
\end{equation}
In a normal parametrization, $h_t=\kappa\Theta$, which proves \eqref{eq:h-norm}.
\end{proof}

At points where the height attains its maximum or minimum, the height equation \eqref{eq:h-norm} gives differential inequalities for the extremal heights.

Define the maximum and minimum heights (see Figure \ref{fig:ext}) by
\begin{equation}
    \bar z(t):=\max_{\Gamma_t}\pi_I,\qquad \underline z(t):=\min_{\Gamma_t}\pi_I.
\end{equation}

\begin{lemma}\label{lem:ext}
For any smooth positive warping function $r$, without assuming graphicality, the height extrema satisfy
\begin{equation}\label{eq:ext}
    \bar{z}'(t)\le -f(\bar z(t)),\qquad \underline{z}'(t)\ge -f(\underline{z}(t))\quad\text{for a.e. }t.
\end{equation}
\end{lemma}

\begin{proof}
On every compact time interval, $h_t$ is uniformly bounded, so $\bar z$ is Lipschitz there and differentiable almost everywhere.
Fix a time $t_*$ at which $\bar z$ is differentiable, and choose $x_*\in S^1$ with $h(x_*,t_*)=\bar z(t_*)$.
With $x_*$ held fixed, the function $\tau\mapsto\bar z(\tau)-h(x_*,\tau)$ is nonnegative and vanishes at $t_*$, so $\bar z'(t_*)=h_t(x_*,t_*)$.
At $(x_*,t_*)$, we have $h_s=0$ and $h_{ss}\le0$, and the unit-speed identity gives $\Theta^2=1$.
Thus \eqref{eq:h-norm} yields $\bar z'(t_*)\le-f(\bar z(t_*))$.
The inequality for $\underline z$ follows by the same argument, using $h_{ss}\ge0$ at a minimum.
\end{proof}

\begin{figure}
    \centering
    \includegraphics[width=0.8\linewidth]{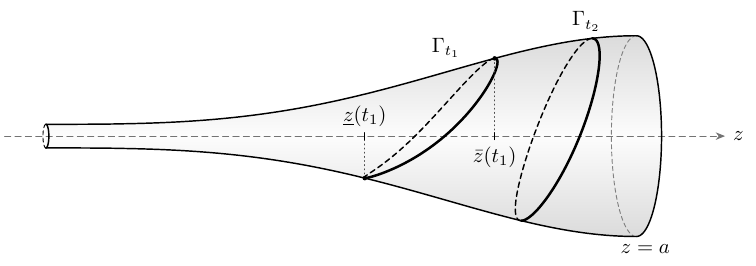}
    \caption{The height extrema $\underline z(t_1)$ and $\bar z(t_1)$ of $\Gamma_{t_1}$, and the curve $\Gamma_{t_2}$ at a time $t_2<t_1$.}
    \label{fig:ext}
\end{figure}

The inequality for the maximum height shows that the condition $r'(a)=0$ of Lemma \ref{lem:slice} holds whenever a closed ancient solution exists, and that the maximum height of such a solution tends to $a$ as $t\to-\infty$.

\begin{lemma}\label{lem:zmax}
The existence of a smooth closed ancient solution implies
\begin{equation}\label{eq:zmax}
    \int_{z_0}^a\frac{d\zeta}{f(\zeta)}=\infty,\qquad r'(a)=0,\qquad \bar{z}(t)\to a\quad(t\to-\infty),
\end{equation}
for any fixed $z_0<a$.
\end{lemma}

\begin{proof}
Fix $t_0<T$.
Since $\bar z$ is locally Lipschitz and $f>0$, integrating \eqref{eq:ext} gives
\begin{equation}\label{eq:g-lb}
    \int_{\bar z(t_0)}^{\bar z(t)}\frac{d\zeta}{f(\zeta)}\ge t_0-t\qquad(t<t_0).
\end{equation}
The right-hand side tends to infinity as $t\to-\infty$.
Since the integral is finite for every upper limit below $a$, this implies $\bar z(t)\to a$ and $\int_{\bar z(t_0)}^a\frac{d\zeta}{f(\zeta)}=\infty$.
Changing the lower limit to any $z_0<a$ adds only a finite quantity.
Finally, the integral criterion in the proof of Lemma~\ref{lem:slice} gives $r'(a)=0$.
\end{proof}

The length satisfies
\begin{equation}\label{eq:l-evol}
    L'(t)=-\int_{\Gamma_t}\kappa^2\,ds.
\end{equation}
Indeed, in a normal parametrization, $\partial_tds=-\kappa^2ds$, and integration over the closed curve proves \eqref{eq:l-evol}.
Thus $L$ is nonincreasing.
If $L(t)\le\Lambda$ for all $t\le t_0<T$, then $L(t)\le L(t_0)\le\Lambda$ for $t_0<t<T$ as well.

\section{Length bounds and graphicality}\label{sec:lg}

In this section, we prove Theorem \ref{thm:main}.

\subsection{Graphicality from length control}\label{sec:lgph}

In Lemmas \ref{lem:strip}--\ref{lem:geod} and Proposition \ref{prop:gph}, we additionally assume that the ancient solution satisfies
\begin{equation}\label{eq:l-bdd}
    L(t)\le \Lambda\qquad \text{for all }t\le t_0,
\end{equation}
for some \(t_0<T\) and \(\Lambda<\infty\).

\begin{lemma}\label{lem:strip}
    For every smooth closed curve \(\Gamma\subset M\),
    \begin{equation}\label{eq:osc}
        \operatorname{osc}_\Gamma z
        :=
        \max_{\Gamma}z - \min_{\Gamma}z
        \le \frac{L(\Gamma)}{2}.
    \end{equation}

    Let $\{\Gamma_t\}_{t<T}$ be an ancient solution whose curves are not null-homotopic in $M$, and suppose that $L(t)\le\Lambda$ for all $t\le t_0<T$.
    Then there exist $b<a$ and $t_1\le t_0$ such that
    \begin{equation}\label{eq:strip}
        \Gamma_t\subset S^1\times [b,a),\quad 0<L_*:=2\pi r(b)\le L(t)\le \Lambda \qquad \text{for all }t\le t_1.
    \end{equation}
\end{lemma}
\begin{proof}
    Let \(s\) denote arclength and write \(h=z|_\Gamma\).
    Since \(\nabla z=\partial_z\) has unit length,
    \[  |h_s|=|\langle\mathbf{T},\partial_z\rangle|\le 1.
    \]
    Choose points where \(h\) attains its maximum and minimum.
    The shorter of the two arcs joining them has length at most \(L(\Gamma)/2\).
    Integrating \(|h_s|\le 1\) along this arc proves \eqref{eq:osc}.

    By Lemma \ref{lem:zmax}, choose \(t_1\le t_0\) so that \(\bar{z}(t)\ge a-1\) for every \(t\le t_1\).
    Then \eqref{eq:osc} gives \(\underline{z}(t)\ge a-1-\Lambda/2=:b\) on this interval.

    For the lower length bound, fix \(t\le t_1\) and write a unit-speed parametrization of \(\Gamma_t\) as \(\gamma(s)=(\vartheta(s),h(s))\), where \(\vartheta:[0,L(t)]\to\mathbb R\) is a lift of the angular coordinate.
    Since $\Gamma_t$ is not null-homotopic in $M$,
    \[
    \vartheta(L(t))-\vartheta(0)=2\pi k\qquad \text{for some }k\in \mathbb{Z}\setminus\{0\}.
    \]
    The identity \(h_s^2+r(h)^2\vartheta_s^2=1\) gives
    \begin{equation*}
        \begin{split}
            L(t)&\ge\int_{\Gamma_t}r(h)|\vartheta_s|\,ds\ge r(b)\int_{\Gamma_t}|\vartheta_s|\,ds\\
            &\ge r(b)\left| \int_{\Gamma_t}\vartheta_s\,ds \right|=2\pi r(b)|k|\ge 2\pi r(b).
        \end{split}
    \end{equation*}
    This proves \eqref{eq:strip}.
\end{proof}

On the strip given by Lemma \ref{lem:strip}, the Gaussian curvature of $M$ is bounded and $L(t)\ge L_*$.
These bounds give a differential inequality for the curvature energy.

\begin{lemma}\label{lem:e-dec}
    Under \eqref{eq:l-bdd}, the curvature energy satisfies
    \begin{equation}\label{eq:e-dec}
    E(t):=\int_{\Gamma_t}\kappa^2\,ds\to 0\qquad \text{as }t\to-\infty.
    \end{equation}
\end{lemma}
\begin{proof}
    Let \(b\) and \(t_1\) be as in Lemma \ref{lem:strip}.
    For \(t\le t_1\), use a normal parametrization of \(\Gamma_t\).
    Since \(K=-r''/r\), the boundedness of $r''$ and \(r\ge r(b)>0\) imply that \(|K|\le K_0\) there.
    The evolution equations (see \cite[Lemmas~1.1--1.2]{Gage90})
    \[
    \partial_t\kappa=\kappa_{ss}+\kappa^3+K\kappa,\qquad \partial_t ds=-\kappa^2 ds
    \]
   give, after integration by parts on the closed curve,
   \begin{equation}\label{eq:e-evol}
       E'(t)=-2\int_{\Gamma_t}\kappa_s^2\,ds+\int_{\Gamma_t}\kappa^4\,ds+2\int_{\Gamma_t}K\kappa^2\,ds.
   \end{equation}
Since $|K|\le K_0$, \eqref{eq:e-evol} gives
\begin{equation}\label{eq:e-ub}
    \begin{split}
        E'(t)\le -2\int_{\Gamma_t}\kappa_s^2\,ds+\int_{\Gamma_t}\kappa^4\,ds+2K_0E(t).
    \end{split}
\end{equation}

For the following spatial estimates, fix $t\le t_1$ and write $L=L(t)$.
All norms are taken on $\Gamma_t$ with respect to $ds$.
Define $\overline{\kappa}:=L^{-1}\int_{\Gamma_t}\kappa\,ds$ and choose $P_0\in\Gamma_t$ such that $\kappa(P_0)=\overline{\kappa}$.
For any $P\in\Gamma_t$, integration along an arc from $P_0$ to $P$ gives
\begin{equation*}
    \begin{split}
        \kappa(P)^2-\overline{\kappa}^2
        =2\int_{P_0}^P\kappa\kappa_s\,ds
        \le 2\int_{\Gamma_t}|\kappa||\kappa_s|\,ds.
    \end{split}
\end{equation*}
By Cauchy--Schwarz and $\overline{\kappa}^2\le L^{-1}\norm{\kappa}_2^2$, we obtain
\[
\norm{\kappa}_{\infty}^2\le \frac{1}{L}\norm{\kappa}_2^2+2\norm{\kappa}_2\norm{\kappa_s}_2.
\]
Using $L\ge L_*$, we obtain
\begin{equation*}
    \begin{split}
\int_{\Gamma_t}\kappa^4\,ds
&\le \norm{\kappa}_{\infty}^2E(t)
\le L_*^{-1}E(t)^2+2E(t)^{3/2}\norm{\kappa_s}_2\\
&\le \norm{\kappa_s}_2^2+L_*^{-1}E(t)^2+E(t)^3.
    \end{split}
\end{equation*}
Substituting this into \eqref{eq:e-ub} yields
\begin{equation}\label{eq:e-poly}
    \begin{split}
        E'(t)\le C(E(t)+E(t)^2+E(t)^3), \qquad t\le t_1,
    \end{split}
\end{equation}
where \(C=C(K_0,L_*)>0\) is independent of \(t\).

The length identity \(L'(t)=-E(t)\) and \eqref{eq:l-bdd} imply that the finite limit
\[
L_{-\infty}:=\lim_{t\to-\infty}L(t)
\]
exists and that
\begin{equation}\label{eq:e-int}
    \int_{-\infty}^{t_1}E(t)\,dt=L_{-\infty}-L(t_1)<\infty.
\end{equation}
Suppose that \eqref{eq:e-dec} fails.
Choose $\varepsilon>0$ and $t_j\le t_1$ with $t_j\to-\infty$ and $E(t_j)\ge\varepsilon$.
Set
\[
\delta_\varepsilon:=\frac{1}{2C}\int_{\varepsilon/2}^{\varepsilon}\frac{dv}{v+v^2+v^3}>0.
\]
Inequality \eqref{eq:e-poly} implies
\begin{equation}\label{eq:e-lb}
    E(t)\ge \frac\varepsilon2\qquad \text{for }t\in[t_j-\delta_\varepsilon,t_j].
\end{equation}
Indeed, if the energy were smaller than \(\varepsilon/2\) somewhere in this interval, continuity would give a subinterval on which it goes from \(\varepsilon/2\) to \(\varepsilon\) while remaining between these two values.
Dividing \eqref{eq:e-poly} by \(C(E+E^2+E^3)\) and integrating on this subinterval shows that its length is at least \(2\delta_\varepsilon\), a contradiction.
It follows from \eqref{eq:e-lb} that
\[
\int_{t_j-\delta_\varepsilon}^{t_j}E(t)\,dt\ge \frac{\varepsilon}{2}\delta_\varepsilon.
\]
The left-hand side tends to zero by \eqref{eq:e-int}, which is impossible.
Therefore \(E(t)\to 0\) as \(t\to-\infty\).
\end{proof}

The strip and length bounds, together with the decay of the curvature energy, give subsequential $C^1$ convergence of normalized arclength parametrizations to a closed geodesic.
We include the standard compactness argument for the reader's convenience.

\begin{lemma}\label{lem:c1}
    Under \eqref{eq:l-bdd}, every sequence \(t_j\to -\infty\) admits a subsequence and positively oriented normalized arclength parametrizations which converge in \(C^1\) to an immersed closed geodesic \(\gamma_\infty\) in \(S^1\times [b,a]\).
    The limiting curve has angular degree \(1\) and satisfies
    \begin{equation}\label{eq:zmax-a}
        \max_{\gamma_\infty}z=a.
    \end{equation}
\end{lemma}

\begin{proof}
    Let $L_j:=L(t_j)$ and choose positively oriented unit-speed parametrizations $\gamma_j:[0,L_j]\to\Gamma_{t_j}$.
    Define
    \begin{equation}\label{eq:norm-par}
        \tilde{\gamma}_j(x)=\gamma_j(L_jx),\qquad 0\le x\le 1.
    \end{equation}
    Write $\gamma_j=(\vartheta_j,h_j)$ in lifted coordinates with $\vartheta_j(0)\in[0,2\pi]$.
    Since $\Gamma_{t_j}$ has degree one, $\vartheta_j(L_j)=\vartheta_j(0)+2\pi$.
    For large $j$, Lemma~\ref{lem:strip} gives a fixed strip and $L_*\le L_j\le\Lambda$, which bound the lifted coordinate maps and their first derivatives uniformly.
    On this strip, the metric is uniformly equivalent to the Euclidean coordinate metric and the Christoffel symbols are bounded, so $\gamma_{j,ss}+\Gamma(\gamma_j)(\gamma_{j,s},\gamma_{j,s})=\nabla_s\gamma_{j,s}=\kappa\mathbf N$ gives
    \[
    \int_0^{L_j}|\gamma_{j,ss}|^2\,ds\le C(E(t_j)+L_j).
    \]
    Rescaling to $[0,1]$ and using Lemma~\ref{lem:e-dec}, we obtain the coordinate bounds
    \begin{equation}\label{eq:w22-bd}
        \norm{\tilde{\gamma}_j}_{W^{2,2}([0,1])}\le C, \qquad |\tilde{\gamma}_{j,x}|_g=L_j\ge L_*>0.
    \end{equation}

    Since $r$ and $r'$ extend continuously to $a$ and $r\ge r(b)>0$, the metric and Christoffel symbols extend continuously to $S^1\times[b,a]$.
    By compact Sobolev embedding, a subsequence satisfies
    \begin{equation}\label{eq:c1-conv}
        \tilde{\gamma}_j\to\gamma_\infty\quad \text{in }C^1([0,1])\quad \text{and weakly in }W^{2,2}([0,1]).
    \end{equation}
    The relations $\tilde\gamma_j(1)=\tilde\gamma_j(0)+(2\pi,0)$ and $\tilde\gamma_{j,x}(1)=\tilde\gamma_{j,x}(0)$ pass to the limit, so $\gamma_\infty$ is a closed curve of degree one.
    Moreover, $\gamma_\infty$ has constant speed $|\gamma_{\infty,x}|_g=L_{-\infty}\ge L_*>0$.

    Rescaling $\nabla_s\gamma_{j,s}=\kappa\mathbf N$ also gives
    \begin{equation}\label{eq:cov2}
        \int_0^1 |\nabla_x\tilde{\gamma}_{j,x}|^2_g\,dx=L_j^3E(t_j)\to 0.
    \end{equation}
    By \eqref{eq:c1-conv} and continuity of the extended Christoffel symbols $\Gamma$, \eqref{eq:cov2} yields
    \[
    \gamma_{\infty,xx}+\Gamma(\gamma_\infty)(\gamma_{\infty,x},\gamma_{\infty,x})=0
    \]
    in the distributional sense.
    The connection term is continuous, so $\gamma_\infty$ is $C^2$ and the equation holds classically, including at height $a$.
    Finally, Lemma~\ref{lem:zmax} and uniform convergence of the height components give \eqref{eq:zmax-a}.
\end{proof}

The limiting geodesic reaches height $a$, and Clairaut's relation shows that its entire image lies in $S^1\times\{a\}$.

\begin{lemma}\label{lem:geod}
    The limiting geodesic \(\gamma_\infty\) in Lemma \ref{lem:c1} traverses the circle \(S^1\times\{a\}\) exactly once.
    In particular,
    \begin{equation}\label{eq:l-geod}
        L_{-\infty}=2\pi r(a).
    \end{equation}
\end{lemma}

\begin{proof}
    Reparametrize the limiting geodesic by arclength and write \(\gamma_\infty(s)=(\vartheta(s),h(s))\).
    The unit-speed condition is
    \begin{equation}\label{eq:unit}
        h_s^2+r(h)^2\vartheta_s^2=1.
    \end{equation}
    Since the metric is independent of the angular coordinate, the $\theta$-component of the geodesic equation gives the Clairaut constant
    \begin{equation}\label{eq:clrt}
        C_0:=r(h)^2\vartheta_s.
    \end{equation}
    Indeed, differentiating the right-hand side and using \(\vartheta_{ss}+2f(h)h_s\vartheta_s=0\) gives zero.

    At a maximum point of \(h\), we have \(h=a\) and \(h_s=0\).
    Equation \eqref{eq:unit} then implies \(|C_0|=r(a)\).
    Substituting \eqref{eq:clrt} into \eqref{eq:unit} yields
    \begin{equation}\label{eq:clrt-id}
        h_s^2=1-\frac{r(a)^2}{r(h)^2}.
    \end{equation}
    Since \(r'>0\), the function \(r\) is strictly increasing, so \(r(h)<r(a)\) whenever \(h<a\).
    The right-hand side of \eqref{eq:clrt-id} would be negative at such a point.
    Therefore \(h\equiv a\).

    The curve has degree \(1\), so its orientation is positive and it traverses this circle exactly once.
    Its length \(L_{-\infty}\) is consequently \(2\pi r(a)\), proving \eqref{eq:l-geod}.
\end{proof}

The limiting parametrization has positive angular derivative, so every curve of degree one that is sufficiently close to it in $C^1$ is graphical.

\begin{proposition}\label{prop:gph}
    Assume \eqref{eq:l-bdd}.
    There exists \(T_0<T\) such that \(\Gamma_t\) is graphical over \(S^1\) for every \(t\le T_0\).
    In the corresponding fixed angular parametrization \(\Gamma_t=\{(\theta,z(\theta,t))\,:\,\theta\in S^1\}\),
    \begin{equation}\label{eq:z-c1}
        \norm{z(\cdot,t)-a}_{C^1(S^1)}\to 0\qquad \text{as }t\to-\infty.
    \end{equation}
\end{proposition}

\begin{proof}
    Suppose that graphicality fails at arbitrarily negative times.
    Choose \(t_j\to-\infty\) such that \(\Gamma_{t_j}\) is not graphical.
    Lemmas \ref{lem:c1} and \ref{lem:geod} give a subsequence of normalized arclength parametrizations converging in \(C^1\) to a parametrization of the circle $S^1\times\{a\}$ traversed once in the positive direction.
    Writing \(E_\theta=r(z)^{-1}\partial_\theta\), we thus have
    \[
    \langle \mathbf{T}_j,E_\theta\rangle>\frac12
    \]
    at every point for all sufficiently large \(j\).
    It follows that
    \[
    d\theta(\mathbf{T}_j)=r(z)^{-1}\langle\mathbf{T}_j,E_\theta\rangle>0.
    \]
    The restriction of the angular projection \(\pi_{S^1}:M\to S^1\) to \(\Gamma_{t_j}\) has positive derivative and degree \(1\), so it is a diffeomorphism.
    This says that \(\Gamma_{t_j}\) is a graph, contradicting the choice of \(t_j\).
    Thus $\Gamma_t$ is graphical for all sufficiently negative $t$.

    By Lemmas~\ref{lem:c1} and~\ref{lem:geod}, every sequence \(t_j\to-\infty\) admits a subsequence with normalized arclength parametrizations \(\tilde\gamma_j=(\vartheta_j,h_j)\) in lifted coordinates satisfying \(h_j\to a\) in \(C^1([0,1])\) and \(\vartheta_{j,x}\to2\pi\) uniformly.
    Since the angular maps are diffeomorphisms, \(z(\cdot,t_j)\to a\) uniformly and
    \[
    z_\theta(\vartheta_j(x),t_j)=\frac{h_{j,x}(x)}{\vartheta_{j,x}(x)}\longrightarrow0
    \]
    uniformly in \(x\).
    Every sequence tending to \(-\infty\) therefore has a subsequence along which the norm in \eqref{eq:z-c1} tends to zero, which proves \eqref{eq:z-c1}.
\end{proof}

Graphicality at sufficiently negative times extends to all later times by the following preservation result of Fujihara \cite{Fujihara24}.

\begin{lemma}
\label{lem:gph-pr}
Let $\{\Gamma_t\}_{t\in(-\infty,T)}$ be a smooth, closed, embedded solution to curve shortening flow whose curves are not null-homotopic in $M$.
If $\Gamma_{t_0}$ is graphical over $S^1$ for some $t_0<T$, then $\Gamma_t$ is graphical over $S^1$ for every $t\in[t_0,T)$.
\end{lemma}

\begin{proof}
Choose the orientation of angular degree one and work in a normal parametrization.
The angle evolution identity in \cite[Lemma~2.2]{Fujihara25}, after reversing its curvature sign to match $\partial_tF=\kappa\mathbf{N}$, gives
\[
(\partial_t-\partial_s^2)\Theta
=
\left[
(\kappa+f(h)\Theta)^2
+\bigl(f'(h)-f(h)^2\bigr)(1-\Theta^2)
\right]\Theta.
\]

Fix $t_1\in(t_0,T)$.
Since the flow is smooth and the parameter domain $S^1\times[t_0,t_1]$ is compact, its image is contained in a compact subset of $M$.
Hence the coefficient multiplying $\Theta$ in the above equation is bounded on this time interval.
The maximum principle and $\Theta(\cdot,t_0)>0$ therefore imply
\[
\Theta(\cdot,t)>0
\qquad\text{for every }t\in[t_0,t_1].
\]

Since $d\theta(\mathbf T)=\Theta/r(h)>0$ and the angular projection has degree one, $\Gamma_t$ is graphical.
As $t_1<T$ was arbitrary, the result follows.
\end{proof}

By Lemma \ref{lem:gph-pr}, failure of graphicality persists backward in time.
We use this to bound the growth of the length from below.

\begin{proposition}\label{prop:subsqrt}
    If the length satisfies
    \begin{equation}\label{eq:l-sub}
        L(t)=o(\sqrt{-t})\qquad \text{as }t\to-\infty,
    \end{equation}
    then \(\Gamma_t\) is graphical over \(S^1\) for every \(t<T\).
\end{proposition}

\begin{proof}
    Suppose that \(\Gamma_{t_*}\) is not graphical for some \(t_*<T\).
    By Lemma \ref{lem:gph-pr}, \(\Gamma_t\) is not graphical for any \(t\le t_*\).
    Define
    \begin{equation}\label{eq:j-def}
        \mathcal{J}:=\langle \mathbf{T},\partial_\theta\rangle =r(h)\Theta.
    \end{equation}
    Differentiating \eqref{eq:j-def} and using \eqref{eq:csf}, we obtain
    \begin{equation*}
        \begin{split}
            \mathcal{J}_s
            &=\langle \nabla_\mathbf{T}\mathbf{T},\partial_\theta \rangle+\langle \mathbf{T},\nabla_\mathbf{T}\partial_\theta \rangle\\
            &=\langle \kappa\mathbf{N},\partial_\theta\rangle =-h_s r(h)\kappa.
        \end{split}
    \end{equation*}
    By \eqref{eq:conn}, the second term in the first line is $\vartheta_s h_s(r(h)^2f(h)-r(h)r'(h))=0$.
    The last equality uses $\langle\mathbf N,\partial_\theta\rangle=-r(h)h_s$, which follows from \eqref{eq:n-fr}.
    Since \(r(h)\le r(a)\) and \(|h_s|\le 1\), we obtain
    \[
    |\mathcal{J}_s|\le r(a)|\kappa|.
    \]
    At each fixed time $t\le t_*$, the angular projection of $\Gamma_t$ has a critical point $P_0$.
    Indeed, if it had no critical point, it would be a degree one covering and hence a diffeomorphism, contrary to nongraphicality.
    Thus $\mathcal J(P_0)=0$.
    Choose $P_1\in\Gamma_t$ where the height attains its maximum, so $h_s(P_1)=0$.
    The unit-speed identity gives
    \begin{equation}\label{eq:j-max}
        |\Theta(P_1)|=1,\qquad |\mathcal{J}(P_1)|=r(\bar{z}(t)).
    \end{equation}
    The two arcs joining $P_0$ and $P_1$ each contribute at least \(r(\bar{z}(t))\) to the total variation of \(\mathcal{J}\).
    Thus, by the Cauchy--Schwarz inequality,
    \begin{equation*}
        \begin{split}
            2r(\bar{z}(t))
            &\le\int_{\Gamma_t}|\mathcal{J}_s|\,ds\le r(a)\int_{\Gamma_t}|\kappa|\,ds\\
            &\le r(a)\sqrt{L(t)\int_{\Gamma_t}\kappa^2\,ds}=r(a)\cdot(L(t)E(t))^{1/2}.
        \end{split}
    \end{equation*}
    Hence
    \[
    L(t)E(t)\ge 4\frac{r(\bar{z}(t))^2}{r(a)^2}.
    \]
    By Lemma \ref{lem:zmax}, \(\bar{z}(t)\to a\) as \(t\to-\infty\).
    Hence there exists \(\tau<t_*\) such that \(r(\bar{z}(t))\ge r(a)/2\) for every \(t\le \tau\).
    On this interval, \(L(t)E(t)\ge 1\), and \eqref{eq:l-evol} gives
    \[
    \frac{d}{dt}L(t)^2=2L(t)L'(t)=-2L(t)E(t)\le -2.
    \]
    Integrating from \(t\) to \(\tau\), we find
    \[
    L(t)^2\ge L(\tau)^2+2(\tau-t)\qquad (t\le\tau).
    \]
    Thus
    \[
    \liminf_{t\to-\infty}\frac{L(t)^2}{-t}\ge 2,
    \]
    contradicting \(L(t)=o(\sqrt{-t})\).
    Therefore \(\Gamma_t\) is graphical for every \(t<T\).
\end{proof}

\subsection{Length bound from graphicality}\label{sec:gl}

In the following two lemmas and Proposition \ref{prop:l-gph}, assume that $\Gamma_t$ is graphical for every $t\le t_0$, where $t_0<T$.
On this interval, reparametrize the curves by the angular coordinate as
\begin{equation}
    \bar F(\theta,t)=(\theta,z(\theta,t)),\qquad \theta\in S^1,\quad t\le t_0.
\end{equation}
The two parametrizations satisfy $F(x,t)=\bar F(\vartheta(x,t),t)$.
Thus $h(x,t)=z(\vartheta(x,t),t)$ and $h_t=z_t+z_\theta\vartheta_t$, with the derivatives of $z$ evaluated at $(\vartheta(x,t),t)$.
The normal velocity of $\bar F$ satisfies
\begin{equation}
    \langle\partial_t\bar F,\mathbf{N}\rangle=\kappa,
\end{equation}
although its tangential velocity need not vanish.
The height function $z$ satisfies
\begin{equation}\label{eq:h-gph}
    \partial_tz=\frac{z_{\theta\theta}}{r(z)^2+z_\theta^2}-f(z)\frac{r(z)^2+2z_\theta^2}{r(z)^2+z_\theta^2}.
\end{equation}
This equation can be rewritten as
\begin{equation}\label{eq:h-div1}
    r(z)z_t=\partial_\theta\arctan\left(\frac{z_\theta}{r(z)}\right)-r'(z),
\end{equation}
\begin{equation}\label{eq:h-div2}
    z_t=\partial_\theta\left(\frac{1}{r(z)}\arctan\left(\frac{z_\theta}{r(z)}\right)\right)-f(z)+f(z)\frac{z_\theta}{r(z)}\arctan\left(\frac{z_\theta}{r(z)}\right).
\end{equation}

To derive \eqref{eq:h-gph}, put $\sigma:=\sqrt{r(z)^2+z_\theta^2}$, so that $ds=\sigma\,d\theta$.
Then
\begin{equation}
\begin{aligned}
    \mathbf{T}&=\frac{1}{\sigma}(\partial_\theta+z_\theta\partial_z),\qquad \Theta=\frac{r(z)}{\sigma}>0,\\
    \mathbf{N}&=-\frac{z_\theta}{r(z)\sigma}\partial_\theta+\frac{r(z)}{\sigma}\partial_z.
\end{aligned}
\end{equation}
In the fixed angular parametrization, the normal velocity condition gives $z_t\Theta=\kappa$.
Using the spatial identity $\kappa\Theta=\Delta z-f(z)\Theta^2$ from the proof of Lemma \ref{lem:height}, we obtain
\begin{equation}\label{eq:zt}
    z_t=\frac{\Delta z-f(z)\Theta^2}{\Theta^2}.
\end{equation}
Since $\partial_s=\sigma^{-1}\partial_\theta$ and $\sigma_\theta=z_\theta(rr'+z_{\theta\theta})/\sigma$, we have
\begin{equation}
    \Delta z=\frac{1}{\sigma}\partial_\theta\left(\frac{z_\theta}{\sigma}\right)=\frac{r^2z_{\theta\theta}-rr'z_\theta^2}{\sigma^4}.
\end{equation}
Substituting this formula into \eqref{eq:zt} proves \eqref{eq:h-gph}.
Finally,
\begin{equation}
    \partial_\theta\arctan\left(\frac{z_\theta}{r}\right)=\frac{rz_{\theta\theta}-r'z_\theta^2}{r^2+z_\theta^2}.
\end{equation}
Combining this identity with \eqref{eq:h-gph} proves \eqref{eq:h-div1}.
The identity $\partial_\theta(r(z)^{-1})=-f(z)z_\theta/r(z)$ then gives \eqref{eq:h-div2}.

\begin{remark}
The integral of $\Delta z$ with respect to arclength $ds$ is zero, but $\int\Delta z\,d\theta$ need not vanish.
\end{remark}

Using \eqref{eq:h-div1}, we show that a weighted area of the region between the graph and the circle $S^1\times\{a\}$ is bounded and nondecreasing in time.
Its derivative has a positive lower bound whenever the graph crosses a fixed height interval.

\begin{lemma}\label{lem:fd}
    For $t\le t_0$, define
    \begin{equation}\label{eq:f-def}
        \mathcal{F}(t):=\int_0^{2\pi}\int_{z(\theta,t)}^ae^{\xi-a}r(\xi)\,d\xi d\theta,
    \end{equation}
    \begin{equation}\label{eq:d-def}
        D(t):=\int_0^{2\pi}e^{z-a}z_\theta\arctan\left( \frac{z_\theta}{r(z)}\right)\,d\theta.
    \end{equation}
    Then,
    \begin{equation}\label{eq:fd}
        0\le\mathcal{F}(t)\le 2\pi r(a), \qquad \mathcal{F}'(t)=D(t)+\int_0^{2\pi}e^{z-a}r'(z)\,d\theta\ge D(t)\ge 0.
    \end{equation}
    In particular,
    \begin{equation}\label{eq:d-int}
        \int_{-\infty}^{t_0}D(t)\,dt<\infty.
    \end{equation}
    Moreover, if \(a_1<a_2<a\) and \(\underline{z}(t)\le a_1<a_2\le\bar{z}(t)\), then
    \begin{equation}\label{eq:d-lb}
        D(t)\ge d_0(a_1,a_2)>0,
    \end{equation}
    where $d_0(a_1,a_2)>0$ depends only on $a_1$, $a_2$, and the fixed geometry.
\end{lemma}
\begin{proof}
    Since \(0<r\le r(a)\),
    \[
    0\le \mathcal{F}(t)\le 2\pi r(a)\int_{-\infty}^ae^{\xi-a}\,d\xi=2\pi r(a).
    \]
    Differentiating $\mathcal{F}(t)$ and using \eqref{eq:h-div1}, we obtain
    \begin{equation*}
        \begin{split}
            \mathcal{F}'(t)
            &=-\int_0^{2\pi}e^{z-a}r(z)z_t\,d\theta\\
            &=-\int_0^{2\pi}e^{z-a}\partial_\theta\arctan\left( \frac{z_\theta}{r(z)}\right)\,d\theta+\int_0^{2\pi}e^{z-a}r'(z)\,d\theta\\
            &=\int_0^{2\pi}e^{z-a}z_\theta\arctan\left( \frac{z_\theta}{r(z)}\right)\,d\theta + \int_0^{2\pi} e^{z-a}r'(z)\,d\theta.
        \end{split}
    \end{equation*}
    The boundary term vanishes by periodicity.
    Both terms in the last line are nonnegative.
    Integrating over $[\tau,t_0]$, with $\tau<t_0$, gives
    \[
    \int_\tau^{t_0}D(t)\,dt\le \mathcal{F}(t_0)-\mathcal{F}(\tau)\le 2\pi r(a).
    \]
    Letting $\tau\to-\infty$ proves \eqref{eq:d-int}.

    We next prove \eqref{eq:d-lb}.
    Set
    \[
    A:=\{\theta\in[0,2\pi] : a_1\le z(\theta,t)\le a_2\}.
    \]
    The periodic graph crosses the height interval \([a_1,a_2]\) at least twice, and hence
    \begin{equation}\label{eq:cross}
        \int_A|z_\theta|\,d\theta\ge 2(a_2-a_1).
    \end{equation}
    Indeed, the function obtained by truncating \(z\) to \([a_1,a_2]\) is periodic, has minimum \(a_1\) and maximum \(a_2\), and therefore has total variation at least \(2(a_2-a_1)\).
    Its derivative agrees almost everywhere with \(z_\theta \mathbf{1}_{\{a_1<z<a_2\}}\), which gives \eqref{eq:cross}.

    Consider the increasing convex function
    \begin{equation}\label{eq:phi}
        \varphi(x):=x\arctan(x/r(a)),\qquad x\ge0.
    \end{equation}
    Since \(r(z)\le r(a)\),
    \[
    z_\theta\arctan\left( \frac{z_\theta}{r(z)}\right)=|z_\theta|\arctan\left(\frac{|z_\theta|}{r(z)}\right)\ge \varphi(|z_\theta|).
    \]
    Extend \(|z_\theta|\mathbf{1}_A\) by zero on the rest of the circle.
    Since \(\varphi(0)=0\), Jensen's inequality on the circle of length \(2\pi\) gives
    \begin{equation*}
        \begin{split}
            D(t)&\ge e^{a_1-a}\int_A\varphi(|z_\theta|)\,d\theta=e^{a_1-a}\int_0^{2\pi}\varphi(|z_\theta|\mathbf{1}_A)\,d\theta\\
            &\ge 2\pi e^{a_1-a}\varphi \left( \frac{1}{2\pi}\int_0^{2\pi}|z_\theta|\mathbf{1}_A\,d\theta\right)\\
            &\ge2\pi e^{a_1-a}\varphi \left( \frac{a_2-a_1}{\pi}\right)\\
            &=2e^{a_1-a}(a_2-a_1)\arctan\left(\frac{a_2-a_1}{\pi r(a)}\right)=:d_0(a_1,a_2)>0.
        \end{split}
    \end{equation*}
    This proves the assertion.
\end{proof}

The lower bound \eqref{eq:d-lb}, the integrability \eqref{eq:d-int}, and the inequalities of Lemma \ref{lem:ext} show that the minimum height tends to $a$.

\begin{lemma}\label{lem:zmin}
    If \(\Gamma_t\) is graphical over \(S^1\) for all \(t\le t_0<T\), then
    \begin{equation}\label{eq:zmin}
        \underline{z}(t)\to a\qquad \text{as } t\to-\infty.
    \end{equation}
    Hence \(z(\cdot,t)\to a\) uniformly on \(S^1\).
\end{lemma}
\begin{proof}
    Suppose that \eqref{eq:zmin} does not hold.
    Then there exist \(b<a\) and a sequence \(t_i\to-\infty\) such that \(\underline{z}(t_i)\le b\).
    Choose \(a_1,a_2\) with \(b<a_1<a_2<a\).
    Set
    \[
    \tau:=\frac12\int_b^{a_1}\frac{d\xi}{f(\xi)}>0.
    \]
    The number $\tau$ is independent of $i$.
    Since $\underline z$ is locally Lipschitz, integrating \eqref{eq:ext} and using $\underline z(t_i)\le b$ gives, for every $i$ and $t\in[t_i-\tau,t_i]$,
    \begin{equation*}
        \int_b^{\underline z(t)}\frac{d\xi}{f(\xi)}\le t_i-t\le\tau<\int_b^{a_1}\frac{d\xi}{f(\xi)}.
    \end{equation*}
    Since $f>0$, this implies $\underline z(t)<a_1$ throughout $[t_i-\tau,t_i]$.
    Lemma \ref{lem:zmax} gives $\bar z(t)>a_2$ at all sufficiently negative times.
    For large $i$, Lemma \ref{lem:fd} and \eqref{eq:d-int} therefore give
    \begin{equation}
    \begin{aligned}
        0<\tau d_0(a_1,a_2)&\le\int_{t_i-\tau}^{t_i}D(t)\,dt\\
        &\le\int_{-\infty}^{t_i}D(t)\,dt\longrightarrow0,
    \end{aligned}
    \end{equation}
    a contradiction.
    Thus $\underline z(t)\to a$.
    Finally,
    \[
    \norm{z(\cdot,t)-a}_{L^\infty(S^1)}=a-\underline{z}(t)\to 0,
    \]
    which proves uniform convergence.
\end{proof}

Uniform convergence of the height alone does not control the length.
Together with this convergence, the integrability of $D$ gives a sequence of times along which the total variation of the height tends to zero, allowing us to determine the backward limit of $L(t)$.

\begin{proposition}\label{prop:l-gph}
    If an ancient solution is a graph over \(S^1\) for every sufficiently negative time, then
    \begin{equation}\label{eq:l-lim}
        \lim_{t\to-\infty}\underline{z}(t)=a,\qquad L_{-\infty}:=\lim_{t\to-\infty}L(t)=2\pi r(a).
    \end{equation}
    In particular,
    \begin{equation}\label{eq:l-ub}
        L(t)\le 2\pi r(a)\qquad \text{for every }t<T.
    \end{equation}
\end{proposition}

\begin{proof}
    The height convergence follows from Lemma \ref{lem:zmin}.
    By \eqref{eq:d-int}, there is a sequence \(t_j\to-\infty\) such that \(D(t_j)\to 0\).
    Since \(\underline{z}(t_j)\to a\), for all sufficiently large \(j\) we have \(z(\theta,t_j)\ge a-\log2\) for every \(\theta\), and hence \(e^{z-a}\ge 1/2\) on the entire graph.
    Define
    \[
    V_j:=\int_0^{2\pi}|z_\theta(\theta,t_j)|\,d\theta.
    \]
    With the function \(\varphi\) in \eqref{eq:phi}, Jensen's inequality now gives
    \begin{equation}\label{eq:d-jlb}
        \begin{split}
           D(t_j)&\ge \frac12 \int_0^{2\pi}\varphi(|z_\theta(\theta,t_j)|)\,d\theta\\
           &\ge \pi \varphi\left( \frac{V_j}{2\pi}\right)=\frac{V_j}{2}\arctan\left( \frac{V_j}{2\pi r(a)}\right).
        \end{split}
    \end{equation}
    The last expression is strictly positive whenever \(V_j>0\) and is an increasing function of \(V_j\).
    Thus \(D(t_j)\to 0\) implies \(V_j\to 0\).

    The induced line element on a graph is \(\sqrt{r(z)^2+z_\theta^2}\,d\theta\).
    The elementary inequalities
    \[
    r(z)\le \sqrt{r(z)^2+z_\theta^2}\le r(z)+|z_\theta|
    \]
    and the monotonicity of \(r\) imply
    \begin{equation}\label{eq:l-sqz}
        2\pi r(\underline{z}(t_j))\le L(t_j)=\int_0^{2\pi}\sqrt{r(z)^2+z_\theta^2}\,d\theta\le 2\pi r(a)+V_j.
    \end{equation}
    Both outer expressions tend to \(2\pi r(a)\), so \(L(t_j)\to 2\pi r(a)\).

    Since \(L(t)\) is nonincreasing by \eqref{eq:l-evol}, the convergence \(L(t_j)\to2\pi r(a)\) gives \(\lim_{t\to-\infty}L(t)=2\pi r(a)\) and \eqref{eq:l-ub} for every \(t<T\).
\end{proof}

Combining Propositions \ref{prop:gph} and \ref{prop:l-gph}, we obtain the following result.

\begin{proposition}\label{prop:lg}
    Let $\{\Gamma_t\}_{t\in(-\infty,T)}$ be a smooth, closed, embedded ancient solution whose curves are not null-homotopic in $M$.
    Then the following are equivalent:
    \begin{enumerate}[label=\textup{(\arabic*)}]
        \item \label{cond:l}
        There are \(t_0<T\) and \(\Lambda<\infty\) such that \(L(t)\le \Lambda\) for every \(t\le t_0\).
        \item\label{cond:gph}
        There is \(t_0<T\) such that \(\Gamma_t\) is graphical over \(S^1\) for every \(t\le t_0\).
        \end{enumerate}
        Under either condition,
        \begin{equation}\label{eq:lg-lim}
            \min_{\Gamma_t}z\to a,\qquad \lim_{t\to-\infty}L(t)=2\pi r(a),\qquad L(t)\le 2\pi r(a)\quad (t<T).
        \end{equation}
        Moreover, in the graphical parametrization at sufficiently negative times, \(z(\cdot,t)\to a\) in \(C^1(S^1)\).
\end{proposition}

\begin{proof}
    Proposition \ref{prop:gph} proves that condition \ref{cond:l} implies condition \ref{cond:gph}, and Proposition \ref{prop:l-gph} proves the converse.
    Under either condition, the two propositions give \eqref{eq:lg-lim} and the asserted \(C^1\) convergence.
\end{proof}

\begin{proof}[Proof of Theorem \ref{thm:main}]
Proposition \ref{prop:l-gph} gives \ref{main:ii} $\Rightarrow$ \ref{main:iii}, and \ref{main:iii} $\Rightarrow$ \ref{main:iv} follows from $\sqrt{-t}\to\infty$.
Proposition \ref{prop:subsqrt} gives \ref{main:iv} $\Rightarrow$ \ref{main:ii}.
Under any of these conditions, Proposition \ref{prop:lg} gives the asserted $C^1$ convergence and length limit.
\end{proof}

\section{Proof of Theorem \ref{thm:rig}: \texorpdfstring{$L^2$}{L2} energy method}\label{sec:proof}

Throughout this section, we consider an ancient solution satisfying either condition of Proposition \ref{prop:lg}.
By that proposition, the solution is graphical at all sufficiently negative times and converges to the limiting circle $S^1\times\{a\}$ in \(C^1\).
Lemma \ref{lem:gph-pr} then gives graphicality at every later time.
We use the fixed angular parametrization \(\bar F(\theta,t)=(\theta,z(\theta,t))\).

Define
\begin{equation}\label{eq:mean}
    m(t):=\frac{1}{2\pi}\int_0^{2\pi}z(\theta,t)\,d\theta,\qquad u(\theta,t):=z(\theta,t)-m(t).
\end{equation}
Then $\int_0^{2\pi}u\,d\theta=0$.
We write
\[
\mean{v}:=\frac{1}{2\pi}\int_0^{2\pi}v\,d\theta, \qquad p:=u_\theta=z_\theta,\qquad q:=\frac{p}{r(z)}.
\]

\begin{corollary}\label{cor:small}
    For every \(\varepsilon>0\), there exists \(T_\varepsilon<T\) such that
    \begin{equation}\label{eq:small}
        |m(t)-a|\le\varepsilon,\qquad \norm{u(\cdot,t)}_{C^1(S^1)}\le \varepsilon,\qquad \text{for every }t\le T_\varepsilon.
    \end{equation}
\end{corollary}

\begin{proof}
    Proposition \ref{prop:lg} gives \(z(\cdot,t)\to a\) in \(C^1(S^1)\) as \(t\to-\infty\).
    Taking the average gives \(m(t)\to a\).
    Since \(u=z-m\) and \(u_\theta=z_\theta\), it follows that \(u(\cdot,t)\to 0\) in \(C^1(S^1)\).
\end{proof}

To use the smallness in Corollary \ref{cor:small}, we separate the linear terms in the evolution equation for $u$.

\begin{proposition}\label{prop:lin}
    The perturbation satisfies
    \begin{equation}\label{eq:lin}
        u_t=r(m)^{-2}u_{\theta \theta}-f'(m)u+\mathcal{N}[u], \qquad \mathcal{N}[u]:=\mathcal{R}[u]-\mean{\mathcal{R}[u]},
    \end{equation}
    where
    \begin{equation}\label{eq:r-def}
        \mathcal{R}[u]:=\left[ \frac{1}{r(m+u)^2+p^2}-r(m)^{-2} \right]u_{\theta \theta}-\rho-f(m+u)\frac{p^2}{r(m+u)^2+p^2},
    \end{equation}
    with the Taylor remainder
    \begin{equation}\label{eq:rf-def}
        \rho:=f(m+u)-f(m)-f'(m)u.
    \end{equation}
    Moreover,
    \begin{equation}\label{eq:m-evol}
        m'(t)=-f(m)+\mean{\mathcal{R}[u]},
    \end{equation}
    and the contribution of the projected remainder to the energy satisfies
    \begin{equation}\label{eq:proj}
        \mean{u\mathcal{N}[u]}=\mean{u\mathcal{R}[u]}.
    \end{equation}
\end{proposition}

\begin{proof}
    Substituting \(z=m+u\) and \(p=u_\theta\) into \eqref{eq:h-gph}, and adding and subtracting the linear terms, gives the exact identity
    \[
    z_t=-f(m)+r(m)^{-2}u_{\theta \theta} -f'(m)u+\mathcal{R}[u].
    \]
    The averages of \(u\) and \(u_{\theta \theta}\) vanish.
    Taking the average therefore gives \(m'=-f(m)+\mean{\mathcal{R}[u]}\).
    Subtracting this equation from the equation for \(z\) proves \eqref{eq:lin}.

    The identity \eqref{eq:proj} follows from $\mean{u}=0$.
\end{proof}

The remainder contains $u_{\theta\theta}$, whereas Corollary \ref{cor:small} gives only $C^1$ smallness.
We therefore estimate its contribution to the energy, which after an integration by parts involves only $u$ and $u_\theta$.
We keep $\rho$ with the linear reaction term $-f'(m)u$ since $f'$ need not have a limit at $a$.

\begin{lemma}\label{lem:nl}
Fix $b<a$.
There exists $C>0$, depending only on $r(b)$ and $\sup_{[b,a)}|r'|$, such that, at any time $t$ for which $b<z(\theta,t)<a$ for all $\theta$ and $\|u(\cdot,t)\|_{C^1(S^1)}\le\varepsilon\le1$, we have
\begin{equation}\label{eq:nl-bd}
 \left|\langle u\mathcal{N}[u]\rangle
          +\langle u\rho\rangle\right|
       \le C\varepsilon\langle u_\theta^2\rangle.
\end{equation}
The constant is independent of $t$ and $\varepsilon$.
\end{lemma}

\begin{proof}
    Define the flux remainder
    \begin{equation}\label{eq:flux}
        \mathcal{H}:=\frac{1}{r(z)}\arctan\left( \frac{p}{r(z)} \right)-r(m)^{-2}p.
    \end{equation}
    Comparing \eqref{eq:h-div2} with the decomposition in Proposition \ref{prop:lin} gives
    \begin{equation}\label{eq:r-sum}
        \mathcal{R}[u]+\rho=\partial_\theta\mathcal{H}+f(z)q\arctan q.
    \end{equation}
    Using \eqref{eq:proj} and integrating by parts, we obtain
    \begin{equation}\label{eq:nl-id}
        \mean{u\mathcal{N}[u]}+\mean{u\rho}=-\mean{p\mathcal{H}}+\mean{uf(z)q\arctan q}.
    \end{equation}

    The mean \(m\) lies between the minimum and maximum of \(z\).
    Since \(r\ge r(b)>0\) on the strip $S^1\times[b,a)$, both \(r^{-1}\) and the derivative of \(r^{-2}\) are bounded there.
    Write
    \[
    \mathcal{H}=\left[\frac{1}{r(z)^2}-r(m)^{-2}\right]p+\frac{1}{r(z)}(\arctan q-q).
    \]
    The mean value theorem and
    \[
    |\arctan q-q|\le\frac{|q|^3}{3}
    \]
    give
    \begin{equation}\label{eq:flux-bd}
        |\mathcal{H}|\le C(|u|+p^2)|p|.
    \end{equation}
    Furthermore, \(f\) is bounded on this strip and \(|q\arctan q|\le q^2\), so
    \[
    |uf(z)q\arctan q|\le C|u|p^2.
    \]
    Since \(|u|,|p|\le\varepsilon\le 1\), these inequalities imply
    \[
    |p\mathcal{H}|+|uf(z)q\arctan q|\le C\varepsilon p^2.
    \]
    Averaging this inequality and using \eqref{eq:nl-id} proves \eqref{eq:nl-bd}.
\end{proof}

We combine this nonlinear estimate with the assumption $\mu>0$ and the Wirtinger inequality to obtain decay of the $L^2$ energy.
Define the $L^2$ energy of $u$ by
\begin{equation}\label{eq:a-def}
    \mathcal{A}(t):=\mean{u^2}=\frac{1}{2\pi}\int_0^{2\pi}u(\theta,t)^2\,d\theta.
\end{equation}

\begin{proposition}\label{prop:a-dec}
    Assume $\mu>0$.
    There exists \(T_0<T\) such that
    \begin{equation}\label{eq:a-dec}
        \A'(t)\le -\frac{\mu}{2}\A(t)\qquad \text{for every }t\le T_0.
    \end{equation}
\end{proposition}

\begin{proof}
    Multiplying \eqref{eq:lin} by $u$, taking the average, and integrating by parts, we obtain
    \begin{equation}\label{eq:a-id}
        \frac12\A'=-r(m)^{-2}\mean{u_\theta^2}-f'(m)\A+\mean{u\mathcal{N}[u]}.
    \end{equation}
    Since $\mean{u}=0$, the Wirtinger inequality on the circle of length $2\pi$ gives
    \begin{equation}\label{eq:wirt}
        \mean{u_\theta^2}\ge\A.
    \end{equation}

    By Lemma \ref{lem:nl}, on a fixed strip $S^1\times[b,a)$ and whenever $\|u\|_{C^1}\le\varepsilon\le1$,
    \[
        \mean{u\mathcal{N}[u]}\le-\mean{u\rho}+C\varepsilon\mean{u_\theta^2}.
    \]
    Using \eqref{eq:rf-def}, we obtain
    \begin{equation}\label{eq:i23}
        -f'(m)\A+\mean{u\mathcal{N}[u]}\le-\mean{u(f(z)-f(m))}+C\varepsilon\mean{u_\theta^2}.
    \end{equation}
Thus,
\begin{equation}\label{eq:a-ub}
    \frac12 \A'\le -(r(m)^{-2}-C\varepsilon)\mean{u_\theta^2}-\mean{u(f(z)-f(m))}.
\end{equation}

Choose
\[
\eta:=\min \left\{\frac{\mu}{8}, \frac{1}{4r(a)^2}\right\}>0.
\]
    Since \(r(z)\to r(a)\) and \(\liminf_{z\to a^-}(r(z)^{-2}+f'(z))=\mu>0\), there exists \(\delta\in(0,1)\) such that
    \begin{equation}\label{eq:h-lb}
        \frac{1}{r(z)^2}+f'(z)\ge \frac{\mu}{2}, \qquad \frac{1}{r(z)^2}\le \frac{1}{r(a)^2}+\eta\qquad (a-\delta<z<a).
    \end{equation}
In particular,
\begin{equation}\label{eq:beta}
    f'(z)\ge \beta :=\frac{\mu}{2}-\frac{1}{r(a)^2}-\eta \qquad (a-\delta<z<a).
\end{equation}
If $a-\delta<z<a$ along the graph, then $m$ and the segment joining $m$ to $z=m+u$ lie in $(a-\delta,a)$.
Thus,
\begin{equation}\label{eq:react}
    u(f(z)-f(m))=u^2\int_0^1f'(m+\lambda u)\,d\lambda\ge \beta u^2.
\end{equation}

Fix the constant \(C\) in Lemma \ref{lem:nl} for \(b=a-\delta\), and choose \(0<\varepsilon_0\le 1\) so that \(C\varepsilon_0\le \eta\).
By Corollary \ref{cor:small}, choose \(T_0<T\) so that the graph lies in \(S^1\times(a-\delta,a)\) and \(\norm{u(\cdot,t)}_{C^1}\le \varepsilon_0\) for every \(t\le T_0\).
Equations \eqref{eq:a-ub} and \eqref{eq:react} therefore yield
\begin{equation}\label{eq:a-p}
    \frac12 \A'\le -(r(m)^{-2}-C\varepsilon_0)\mean{u_\theta^2}-\beta \A.
\end{equation}
Since \(r(m)\le r(a)\),
\[
r(m)^{-2}-C\varepsilon_0\ge \frac{1}{r(a)^2}-\eta>0.
\]
Applying \eqref{eq:wirt} in \eqref{eq:a-p}, we obtain
\begin{equation*}
    \begin{split}
        \frac12 \A'
        &\le -(r(m)^{-2}+\beta-C\varepsilon_0)\A\\
        &\le -\left( \frac{\mu}{2}-\eta-C\varepsilon_0\right)\A\\
        &\le -\left( \frac{\mu}{2}-2\eta \right)\A \le -\frac{\mu}{4}\A.\qedhere
    \end{split}
\end{equation*}
\end{proof}

Since $\mathcal A(t)\to0$ as $t\to-\infty$, this decay implies that $u$ vanishes at all sufficiently negative times.

\begin{proof}[Proof of Theorem \ref{thm:rig}]
Let $T_0$ be as in Proposition \ref{prop:a-dec}.
Fix an arbitrary $\tau\le T_0$.
Integrating \eqref{eq:a-dec} over $[t,\tau]$, with $t<\tau$, gives
\begin{equation}\label{eq:anc}
 0\le \A(\tau)\le
    e^{-\frac{\mu}{2}(\tau-t)}\A(t).
\end{equation}
Corollary \ref{cor:small} gives $\A(t)\to0$ as $t\to-\infty$.
Letting $t\to-\infty$ with $\tau$ fixed in \eqref{eq:anc}, we obtain $\A(\tau)=0$.
Since $\tau\le T_0$ was arbitrary, $u\equiv0$ for every $t\le T_0$.
Thus the solution is a slice at all sufficiently negative times.

By Lemma~\ref{lem:slice} and uniqueness for smooth curve shortening flow (see \cite[Section~1]{Fujihara25}), $\Gamma_t$ remains a slice for every $T_0\le t<T$.

Write $\Gamma_t=S^1\times\{Z(t)\}$.
Then the height equation gives
\[
 Z'=-f(Z)=-\frac{r'(Z)}{r(Z)}.
\]
By Theorem \ref{thm:main}, $Z(t)\to a$ as $t\to-\infty$.
Since $Z(t)<a$ and $r$ is strictly increasing,
\[
 L(t)=2\pi r(Z(t))<2\pi r(a).\qedhere
\]
\end{proof}

\section*{AI disclosure}
During the preparation of this work, the authors used ChatGPT (OpenAI) to improve the organization and language of the manuscript and to create Figures~\ref{fig:slice} and~\ref{fig:ext}.
After using this tool, the authors reviewed and edited the content as needed and take full responsibility for the content of the publication.

\end{document}